\documentclass[reqno,11 pt]{amsart}

\usepackage{cite}
\usepackage{enumerate}
\usepackage{times}
\usepackage{pifont}
\usepackage{amsmath}
\usepackage{amsfonts}
\usepackage{amsthm}
\usepackage{amssymb}
\usepackage{graphicx}
\usepackage{graphics}
\usepackage{geometry}
\usepackage{latexsym}
\usepackage{hyperref}
\usepackage{color}
\definecolor{dark-blue}{rgb}{0,0,0.6}
\definecolor{Purple}{rgb}{0.2,0,0.25}
\hypersetup{breaklinks=true,
pagecolor=white,
colorlinks=true,
citecolor=dark-blue,
filecolor=black,%
linkcolor=Purple,%
urlcolor=black
}
\newcommand{\bref}[1]{\textbf{\ref{#1}}} 
\newcommand{\beqref}[1]{\textbf{(\ref{#1})}} 

\theoremstyle{theorem}
\newtheorem{theorem}{Theorem}
\newtheorem{lemma}[theorem]{Lemma}
\newtheorem{proposition}[theorem]{Proposition}
\newtheorem{fact}[theorem]{Fact}

\theoremstyle{definition}
\newtheorem{definition}[theorem]{Definition}

\newtheorem{example}[theorem]{Example}

\newcommand{\Int}{\textnormal{Int}}
\newcommand{\Ext}{\textnormal{Ext}}
\newcommand{\R}{\mathbb{R}}
\newcommand{\Q}{\mathbb{Q}}

\newcommand{\Z}{\mathbb{Z}}
\newcommand{\N}{\mathbb{N}}
\newcommand{\Lra}{\Longrightarrow}


\begin{document}
\date{January 30, 2019}\vspace*{-0.6cm}
\title[Fun with ``Analysis I'']{Fun with ``Analysis I'':  basic theorems in calculus revisited}
\author{Daniel Reem}
\address{Daniel Reem, Department of Mathematics, The Technion - Israel Institute of Technology, 3200003 Haifa, Israel.} 
\email{\bf \small dream@technion.ac.il}
\maketitle
 
\vspace*{-0.6cm}
\begin{abstract}
This note tries to show that a re-examination of a first course in analysis, using the more sophisticated tools and approaches obtained in later stages, can be a real fun for experts, advanced students, etc. We start by going to the extreme, namely we present two proofs of the Extreme Value Theorem: ``the programmer proof'' that suggests a method (which is practical in down-to-earth settings) to approximate, to any required precision, the extreme values of the given function in a metric space setting, and an abstract space proof (``the level-set proof'') for semicontinuous functions defined on compact topological spaces.  Next, in the intermediate part, we consider the Intermediate Value Theorem, generalize it to a wide class of discontinuous functions, and re-examine the meaning of the intermediate value property. The trek reaches the final frontier when we discuss the  Uniform Continuity Theorem, generalize it, re-examine the meaning of uniform continuity, and find the optimal delta of the given epsilon. Have fun!
\end{abstract}\vspace*{0.4cm}

A first course in analysis is not always a pleasant experience for fresh students.  However, once the mathematical foundations become firmer, looking back at this first course and re-examining parts of its  material, using the more sophisticated tools and ways of thinking which have been acquired in later stages, can be a real fun for advanced students, experts (teachers, researchers, enthusiasts, etc.), and many others who like mathematics. The goal of this note is to achieve something in this direction by playing with, and looking for new horizons in, three fundamental theorems in calculus and related material. 

We start the trilogy in Section \bref{sec:ExtremeValue} by going to the extreme. More precisely, we  discuss   the Extreme Value Theorem concerning the extreme (optimal) values of a continuous function defined on a compact space. Two short proofs of this theorem are presented. The first is ``the programmer proof''  for functions defined on a compact metric space. This proof, which is presented in Subsection \bref{subsec:ProgrammerProof}, does not follow the path of most of other proofs which are focused on the abstract existence of the extreme  values, but usually do not present any clue regarding estimating these values. Instead, the programmer proof  suggests a method to approximate, to any required precision, the extreme values of the given function and, as a by-product, proves their existence. The method, which, as implied by its name, is in the spirit of programming, is practical in down-to-earth settings, as explained in Subsection \bref{subsec:DownToEarth}. In Subsection \bref{subsec:LevelSetProof} we return back to the abstract space and present  the ``level-set proof'' for semicontinuous functions defined on a general compact topological space and having values in a fully ordered set. Despite the somewhat abstract setting, this proof seems to be natural and guided directly from the definitions. Both  proofs do not use the frequently used argument of proving first that the supremum and infimum of the function are finite, and then proving that they are attained.  

Next we proceed to the intermediate section (Section \bref{sec:IntermediateValue}) where, naturally, the Intermediate Value Theorem is considered. We generalize this theorem to a class of discontinuous functions 
and re-examine the meaning of the intermediate value property. 

The trek reaches the final frontier in Section \bref{sec:UniformContinuity} with a discussion on uniform continuity. We first consider the question of whether the optimal delta of the given epsilon (from the definition of uniform continuity) can be presented explicitly. A new hope emerges in Subsection \bref{subsec:NewHope} after formulating a quantitative necessary and sufficient condition for a function acting  between two metric spaces to be uniformly continuous. Using this condition, the optimal delta is found and a  few basic properties of it are derived. The compactness strikes back in Subsection \bref{subsec:Compactness} 
when we prove, using the optimal delta, the Uniform Continuity Theorem which says that a continuous function which acts between a compact metric space and a metric space must be uniformly continuous. Actually, we prove a more general result in which various sufficient conditions for the uniform continuity of the given function are formulated, including ones in which the function is not  assumed in advance to be continuous. Finally (Subsection \bref{subsec:Semicontinuity}), 
we discuss the question of whether the optimal delta is a  continuous function of epsilon, and this discussion marks the return of the semicontinuity.  Have fun!

\section{Going to the extreme}\label{sec:ExtremeValue}
In its simplest form, the Extreme Value Theorem, which is sometimes called the Weierstrass Theorem, says that a real continuous function $f$ defined on a closed and bounded interval $I\subseteq \R$ attains extreme (optimal) values on the interval. In other words, there are points $x_{\textnormal{min}}$ (a minimizer) and $x_{\textnormal{max}}$ (a maximizer) in $I$ which  satisfy $f(x_{\textnormal{min}})\leq f(x)\leq f(x_{\textnormal{max}})$ for every $x\in I$. This theorem has been generalized to real continuous functions defined on closed and bounded subsets of finite-dimensional Euclidean spaces, to real continuous functions defined on compact metric spaces, and even to semicontinuous functions defined on compact topological spaces and having values in a linearly ordered set. See, e.g., \cite{Bernau1967jour},\cite[p. 129]{DenceDence2010book}, \cite[p. 18]{DunfordSchwartz1958book},\cite{Ferguson2014jour},\cite[pp. 60-61]{Fitzpatrick2006book},\cite{Fort1951jour},\cite[pp. 193-196]{Hardy1967book},\cite[pp. 283-284]{JacobEvans2016book},\cite{Jungck1963jour}, \cite[pp. 190-191]{Kitchen1968book}, \cite[p. 41]{Kothe1969book},\cite{Martinez-Legaz2014jour},\cite[p. 174]{Munkres2000book},\cite{Pennington1960jour},\cite[p. 89]{Rudin1976book}, \cite{Tandra2015jour} and \cite[pp. 236-237]{Tao2016bookI}. 

In this section we discuss two additional proofs of the Extreme Value Theorem: ``the programmer proof'' 
(Subsections \bref{subsec:ProgrammerProof}--\bref{subsec:DownToEarth} below) and ``the level-set proof''  
(Subsection \bref{subsec:LevelSetProof} below). Before presenting these proofs, we note that for us (here and  elsewhere) any  space that we consider (metric or topological) is nonempty by definition. 

\subsection{\bf Dawn: the programmer proof}\label{subsec:ProgrammerProof} 
The idea behind the proof is simple: we make a discretization (digitization) of the space, i.e.,  we approximate it by a finite set of points (which we interpret as the ``digital world''), with the hope that by a better and better approximation, the extreme values of our function over the digital world will better approximate the supremum and infimum of the function over the entire (``continuous'') space. The existence of an arbitrary good discretization is nothing but a reformulation of the well-known and simple  fact that a compact metric space $(X,d)$ is totally bounded; in other words, $(X,d)$  has an $\epsilon$-net for each $\epsilon>0$, i.e., a nonempty finite set $F_{\epsilon}$ of $X$ with the property that for every $x\in X$ there exists $z\in F_{\epsilon}$ such that $d(x,z)<\epsilon$. See, e.g., \cite[p. 885]{ThomsonBrucknerBruckner2008book} or \cite[pp. 275--276]{Munkres2000book}. 

\begin{theorem}\label{thm:Weierstrass}
Let $(X,d)$ be a compact metric space and let $f:X\to \mathbb{R}$ be  continuous. Then $f$ attains both a minimum and a maximum value on $X$.
\end{theorem}
\begin{proof}
Consider an increasing sequence $(E_n)_{n=0}^{\infty}$ of finite subsets of $X$ which is dense in $X$, that is, for every $x\in X$ and $\epsilon>0$ there is some $z\in \bigcup_{n=0}^{\infty}E_n$ such that $d(x,z)<\epsilon$. Such a sequence can be constructed using the fact mentioned above about $\epsilon$-nets. Indeed, let $(\epsilon_n)_{n=0}^{\infty}$ be any decreasing sequence of positive numbers tending to zero, say $\epsilon_n:=2^{-n}$, $n\in\N\cup\{0\}$. The above-mentioned fact implies that for each $n\in\N\cup\{0\}$ there exists an $\epsilon_n$-net of $X$, and we denote it by $F_n$. Now let $E_0:=F_0$ and define by induction $E_{n+1}:=E_n\cup F_n$ for all $n\in\N\cup\{0\}$. Then $(E_n)_{n=0}^{\infty}$ is increasing and we have $\cup_{n=0}^{\infty}E_n=\cup_{n=0}^{\infty}F_n$. To see that $\cup_{n=0}^{\infty}F_n$ is dense, let $\epsilon>0$ and $x\in X$ be arbitrary. We can find $n\in\N$ sufficiently large such that $\epsilon_n<\epsilon$; since $F_n$ is an $\epsilon_n$-net, there is some $z\in F_n\subseteq \cup_{k=0}^{\infty}F_k$ such that $d(x,z)<\epsilon_n<\epsilon$, as claimed. 

For each $n\in\N\cup\{0\}$, let 
\begin{equation}\label{eq:M_n m_n}
\begin{array}{c}
M_n:=\max\{f(x):x\in E_n\},\\
m_n:=\min\{f(x):x\in E_n\}.
 \end{array}
\end{equation}
Because $E_n$ is finite for all $n\in\N\cup\{0\}$, it follows that $f$ attains its maximum on   $E_n$, namely, there exists $x_n\in
E_n$ such that $f(x_n)=M_n$. Let $(x_{n_j})_{j=0}^{\infty}$ be any convergent subsequence of $(x_n)_{n=0}^{\infty}$ whose existence is guaranteed because $X$ is compact. Let 
 $x_{\infty}:=\lim_{j\to \infty}x_{n_j}\in X$. Since $f$ is continuous, we have $f(x_{\infty})=\lim_{j\to \infty}f(x_{n_j})=\lim_{j\to \infty}M_{n_j}$.

Actually, the whole sequence $(M_n)_{n=0}^{\infty}$ converges to
$f(x_{\infty})$ since $(M_n)_{n=0}^{\infty}$ is an increasing sequence (because $E_n\subseteq E_{n+1}$ for all $n\in\N\cup\{0\}$) and it has a subsequence which converges to $f(x_{\infty})$. In particular, $M_n\leq f(x_{\infty})$ for all $n\in\N\cup\{0\}$. It remains to show that $x_{\infty}$ is a maximizer. 
Fix arbitrary $x\in X$ and $\epsilon>0$. Since $f$ is continuous on $X$, it
is continuous at $x$. Hence there exists $\delta>0$ such that given $y\in X$, if $d(y,x)<\delta$, then $|f(x)-f(y)|<\epsilon$. The construction and properties of $(E_n)_{n=0}^{\infty}$ implies that for each $n\in\N$  sufficiently large (such that $\epsilon_n<\delta$), there exists $t_n\in F_n\subseteq E_n$ such  that $d(t_n,x)<\delta$. Therefore 
\begin{equation}\label{f(x)<=M_n+epsilon}
f(x)<f(t_n)+\epsilon\leq M_{n}+\epsilon\leq f(x_{\infty})+\epsilon.
\end{equation}
Since $\epsilon$ was arbitrary, we have $f(x)\leq f(x_{\infty})$, as required. 
A similar consideration (now using $(m_n)_{n=0}^{\infty}$ from \beqref{eq:M_n m_n}) shows that $f$ has a minimizer in $X$.
\end{proof}

\subsection{The programmer proof: down-to-earth}\label{subsec:DownToEarth}
The programmer proof not only proves the existence of extreme values of $f$, it also suggests a method to compute them approximately to any desired precision. Indeed, as is well-known and will be proved in Section \bref{sec:UniformContinuity}, since $f$ is continuous and $X$ is compact, $f$ is actually uniformly continuous on $X$. Now, given  $\epsilon>0$, let $\delta>0$ be any delta from the definition of uniform continuity of $f$ on $X$, say the optimal one defined in \beqref{eq:delta_f} below (see also Examples \bref{ex:delta_f}--\bref{ex:delta_f Discontinuous}). Let $n\in\N$ be sufficiently large such that $F_n$ (and hence also $E_n$) from the proof of Theorem \bref{thm:Weierstrass} forms a $\delta$-net of $X$. Let $M_n$ be defined in \beqref{eq:M_n m_n} and choose an arbitrary $x_n\in E_n$ which satisfies $f(x_n)=M_n$. Since $E_n$ is finite, we can compute both $M_n$ and $x_n$ directly, possibly by brute force, namely, by going over all the values $f(x)$, $x\in E_n$ and finding the maximal value (the computation may be demanding for large $n$). Since $E_n$ is a $\delta$-net of $X$, an argument similar to \beqref{f(x)<=M_n+epsilon} shows that $M_n=f(x_n)$ is an $\epsilon$-approximate  maximal value of $f$ and $x_n$ is an $\epsilon$-approximate maximizer $x_n$ of $f$ (i.e., $|f(x_n)-\max\{f(x): x\in X\}|<\epsilon$). One can say similar things regarding the minimal value of $f$. 

In order to implement the method described in the programmer proof in a computer, one  should be able to produce the digital world sequence $(E_n)_{n=0}^{\infty}$. This is possible in down-to-earth settings. Indeed, suppose for instance that  $X=[a,b]$ for $a,b\in \R$, $a\leq b$.  Then we can define for each $n\in \mathbb{N}\cup\{0\}$, 
\begin{equation}\label{eq:E_nInterval}
E_n:=\{p_k: k\in\{0,1,\ldots,2^n\}\},\,\,\textnormal{where}\,\, p_k:=a+\frac{(b-a)k}{2^n},\, k\in\{0,1,\ldots,2^n\}.
\end{equation}
Similarly, for a box $X=\prod_{i=1}^m [a_i,b_i]$ 
contained in $\R^m$, $m\in\N$ (where $a_i\leq b_i$ for each  $i\in\{1,\ldots,m\}$)  we can take 
\begin{equation}\label{eq:E_nBox}
E_n:=\{p_k:=(p_{k,i})_{i=1}^m: k\in\{0,1,\ldots,2^n\}\},
\end{equation}
where 
\begin{equation}
p_{k,i}:=a_i+\frac{(b_i-a_i)k}{2^n}, \,\,k\in\{0,1,\ldots,2^n\},\,\, i\in\{1,\ldots,m\}.
\end{equation}
Of course, when the dimension $m$ grows, the number of points $p_{k,i}$ grows exponentially with $m$, so this type of approximation process seems to be useful only in low (down-to-earth) dimensions. Anyhow, given a compact metric space $(X,d)$, since we already know that $f$ attains its extreme values on $X$, ideas similar to the ones used in the  programmer proof can be used to show that $(E_n)_{n=0}^{\infty}$ can be taken to be any  sequence of finite subsets of $X$ whose union is dense in $X$, where $M_n$ and $m_n$ are still defined by \beqref{eq:M_n m_n}.

Despite the fact that the programmer proof enables one to approximate the extreme values to any desired precision, it does not give sufficient information to locate  the exact maximizers and minimizers of $f$. Nevertheless, if some additional information is known about $f$, then we can say more regarding these points. For instance, suppose that $f$ has a unique maximizer $x_{\infty}$. Then $f(x_{\infty})=\sup\{f(x): x\in X\}=:M$. We claim that in this case it must be that $\lim_{n\to\infty}x_n=x_{\infty}$ (where $x_n$ is defined after \beqref{eq:M_n m_n}).  Indeed, if this is not true, then for some neighborhood $U$ of $x_{\infty}$ and for some subsequence $(x_{n_k})_{k=1}^{\infty}$ we have $x_{n_k}\notin U$ for all $k\in \N$. Now, because $X$ is compact this subsequence has a subsequence $(x_{n_{k_j}})_{j=1}^{\infty}$ which converges to some point $y_{\infty}\in X$ which is outside $U$. In particular, $y_{\infty}\neq x_{\infty}$. Because we know from the programmer proof that $\lim_{n\to\infty}f(x_n)=\lim_{n\to\infty}M_n=M$, we have $\lim_{j\to\infty}f(x_{n_{k_j}})=M$. But $f$ is continuous and hence $\lim_{j\to\infty}f(x_{n_{k_j}})=f(y_{\infty})$. Thus $f(y_{\infty})=M$, that is, $y_{\infty}$ is a maximizer of $f$. Since $y_{\infty}\neq x_{\infty}$, this is a contradiction to the assumption that $f$ has a unique maximizer. Hence indeed $\lim_{n\to\infty}x_n=x_{\infty}$. 

Again, down-to-earth settings ensuring that $f$ has a unique maximizer/minimizer on $X$ are of interest here. A typical and well-known such a setting for the existence of a maximizer  is when $X$ is a compact convex subset of a normed space and $f$ is  strictly concave, namely, $f:X\to \R$ and $f$ satisfies the inequality $f(\lambda x+(1-\lambda)y)>\lambda f(x)+(1-\lambda)f(y)$ for every $x,y\in X$, $x\neq y$ and every  $\lambda\in (0,1)$. A more general but still not  too abstract such a setting is when $X$ is a compact geodesic metric space and $f$ is strictly quasi-concave. More precisely, by saying that $X$ is a geodesic metric space we mean that between every pair of points in $X$ there is a geodesic segment, that is, given $x,y\in X$, there is a distance preserving mapping $\gamma_{x,y}$ which maps  a  real line segment $[r_1,r_2]$ to $X$ such that $\gamma_{x,y}(r_1)=x$ and $\gamma_{x,y}(r_2)=y$; the geodesic segment associated with $x,y$ and $\gamma_{x,y}$ is the image  $\gamma_{x,y}([r_1,r_2])$; many familiar manifolds are geodesic metric spaces, for example, the Euclidean sphere in which a geodesic segment that connects two points is the shortest part of a large circle on which these points are  located. By saying that $f$ is strictly quasi-concave we mean that  $f(z)>\min\{f(x),f(y)\}$ for all $x,y\in X$, $x\neq y$ and all $z\notin\{x,y\}$ which belongs to  a  geodesic  segment which connects $x$ and $y$. Similarly, if $f$ is strictly convex (that is, $-f$ is strictly concave) and $X$ is a compact convex subset of a normed space, or, more generally, $f$ is strictly quasi-convex (i.e., $-f$ is strictly quasi-concave) and $X$ is a compact geodesic metric space, then $f$ has a unique minimizer on $X$, and a discussion similar to the above one shows that the minimizing sequence from the programmer proof converges to this unique minimizer. 
 
Methods for finding optimal values and optimal points of functions, in various settings, are usually dealt with in optimization theory, e.g., in  \cite{AvrielDiewertSchaibleZang2010book,Ben-IsraelBen-TalZlobec1981book,BonnansGilbertLemarechalSagastizabal2006book,BorweinLewis2006book,CambiniMartein2009-book,CensorZenios1997book,Dantzig1963book,Dixit1990book, Nesterov2004book,Rockafellar1970book}. A significant part of this very rich theory is devoted to  convex and concave functions. The method described in the programmer proof enriches further this theory to abstract and down-to-earth settings.  

\subsection{The level-set proof: back to the abstract space}\label{subsec:LevelSetProof}
We now turn to the level-set proof of the Extreme Value Theorem. While this proof may be considered as being somewhat abstract at first glance, it seems to us (at least in retrospective) rather natural because it emphasizes the key players involved in the theorem: an order relation in the range which forces a simple  formulation of the condition of being an extreme value in terms of an intersection of subsets, a mean (namely, semicontinuity) which ensures that the subsets involved in the intersection are well-behaved, and a  criterion which ensures that the intersection is nonempty. Before presenting the proof, we need to recall a few basic definitions and facts. 
\begin{definition}\label{def:Order}
Let  $(L,\leq)$ be a partially ordered set, namely $L$ is a nonempty set and $\leq$ is a partial order relation on $L$. We say that  $(L,\leq)$ is linearly ordered (or fully ordered, or simply ordered) whenever any two elements $\alpha,\beta\in L$ can be compared: either $\alpha\leq \beta$ or $\beta\leq \alpha$. The order topology $\mathcal{T}_L$ on $L$ is the topology generated by the sets $I_{<\alpha}:=\{\beta\in L: \beta<\alpha\}$ and $I_{>\alpha}:=\{\beta\in L: \beta>\alpha\}$, $\alpha\in L$ (called open rays). The triplet $(L,\leq, \mathcal{T}_L)$ is called a linearly ordered topological space. 
\end{definition}
A few important and familiar examples of fully ordered sets are: $\Z$, $\Q$, $\R$, $(-\infty,\infty]$, $[-\infty,\infty)$, and $[-\infty,\infty]$, all of them with the standard order relation between real numbers (or between them and $\pm\infty$). Another example: $\R^m$ with the dictionary (lexicographic) order, $m\in\N$.  Details about the order topology can be found in various sections of  \cite{Munkres2000book} (e.g., Sections 14, 16, 17, 18 and 24). A useful property of fully ordered sets that we need below can be verified immediately: any finite set $S\neq\emptyset$ of  a fully ordered set $L$  has both a least and a greatest element, namely elements $m$ and $M$ such that $m\leq \alpha\leq M$ for each $\alpha\in S$.   
\begin{definition}\label{def:SemiContinuity}
Given a topological space $(X,\mathcal{T})$, a linearly ordered topological space $(L,\leq,\mathcal{T}_L)$, and a function  $f:X\to L$, we say that $f$ is lower semicontinuous if for every $\alpha\in L$ the  $\leq$-level-set  $f_{\leq\alpha}:=\{z\in X: f(z)\leq \alpha\}$ is closed in $X$ (equivalently, $f^{>\alpha}:=\{z\in X: f(z)>\alpha\}$ is open).  We say that $f$ is upper semicontinuous if for every $\alpha\in L$ the $\geq$-level-set $f^{\geq\alpha}:=\{z\in X: f(z)\geq \alpha\}$ is closed in $X$ (equivalently, $f_{<\alpha}:=\{z\in X: f(z)<\alpha\}$ is open). 
\end{definition}
 It is straightforward to check that if $L$ is endowed with the order topology, then $f:X\to L$ is continuous if and only if it  is both lower and upper semicontinuous.  
\begin{definition}
A set $\mathcal{F}$ whose elements are nonempty sets  is said to have the finite intersection property whenever the intersection of any finitely many members of $\mathcal{F}$ is nonempty. 
\end{definition}
\begin{fact}\label{fact:FiniteIntersectionProperty}
A topological space $(X,\mathcal{T})$ is compact if and only if for each set $\mathcal{F}$ which consists  of nonempty closed subsets of $X$ and has the finite intersection property, the intersection of all the members of $\mathcal{F}$ is nonempty  (see \cite[pp. 169-170]{Munkres2000book} for the immediate proof). 
\end{fact}

\begin{theorem}\label{thm:ExtremeValueSemicontunuity}
Let $(X,\mathcal{T})$ be  a compact topological space and $(L,\leq,\mathcal{T}_L)$ be a linearly  ordered topological space. If $f:X\to L$ is lower semicontinuous, then it attains a minimum, and if $f$ is upper semicontonuous, then it attains a maximum. In particular, if $f$ is continuous, then it has a minimizer and a maximizer in $X$.
\end{theorem}
\begin{proof}
Suppose first that $f$ is lower semicontinuous. Our goal is to prove that $f$ has a minimizer, namely a point $x_*$ having the property that $f(x_*)\leq f(x)$ for all $x\in X$. In other words, $x_*$ should belong to the $f_{\leq f(x)}$-level-sets $J_x:=\{z\in X: f(z)\leq f(x)\}$ for each $x\in X$. Equivalently, $x_*\in \bigcap_{x\in X}J_x$. So it is sufficient and necessary to prove that $\bigcap_{x\in X}J_x\neq\emptyset$. Because our space $X$ is compact, Fact \bref{fact:FiniteIntersectionProperty} ensures that $\bigcap_{x\in X}J_x\neq\emptyset$ once we are able to show that the elements of the set  $\mathcal{F}:=\{J_x: x\in X\}$ are nonempty  closed subsets of $X$ and $\mathcal{F}$ has the finite intersection property. Given $x\in X$, we have $x\in J_x$ and hence $J_x\neq\emptyset$. In addition, $J_x$ is closed because $f$ is lower semicontinuous. As for the finite intersection property, consider an arbitrary finite collection  $\{J_{x_i}: i\in \{1,\ldots,n\}\}$, $n\in\N$ of members of $\mathcal{F}$. Since the set $\{f(x_i): i\in \{1,\ldots,n\}\}$ is a finite set of elements in the fully ordered set $L$, there exists at least one index $i_{\textnormal{min}}\in \{1,\ldots,n\}$ such that $f(x_{i_{\textnormal{min}}})=\min\{f(x_i): i\in \{1,\ldots,n\}\}$. It is immediate to verify that $\cap_{i=1}^n J_{x_i}=J_{x_{i_{\textnormal{min}}}}$. Therefore the intersection is nonempty, as required. The proof in the case where $f$ is upper semicontinuous follows a similar reasoning, where now we re-define $J_x:=\{z\in X: f(x)\leq f(z)\}$ for all $x\in X$. 
\end{proof}

The level-set proof was inspired by the proof of K\"othe for a less general statement \cite[p. 41]{Kothe1969book} (e.g., the range of $f$ there is $[-\infty,\infty]$ and not a general linearly  ordered topological space). K\"othe's proof, while containing important components of the above-mentioned proof, seems to be somewhat obscure and not very natural, e.g., because it is based on the  theory of filters, it does not emphasize the involved key players as done above, and the setting is a compact Hausdorff topological  space (the whole discussion of compactness in \cite{Kothe1969book} is restricted to Hausdorff spaces, and this is apparent even in the definition of compact spaces \cite[p. 16]{Kothe1969book}). Perhaps the main contribution of the level-set proof is to refine the main ideas in K\"othe's proof so that the end result will be more accessible, more natural, more illuminating.

\section{Intermediate time}\label{sec:IntermediateValue}
In its classical 1D form, the Intermediate Value Theorem can be written as follows: 
\begin{theorem}\label{thm:ClassicInter}
Let $I=[a,b]\subset \R$. If $f:I\to \mathbb{R}$ is continuous and if $\gamma\in
\mathbb{R}$ is between $f(a)$ and $f(b)$, then there exists $x\in I$
such that $f(x)=\gamma$.
\end{theorem}
Familiar proofs of either Theorem \bref{thm:ClassicInter} or its traditional generalization saying that a continuous function maps a connected topological space to a connected topological space are heavily based on the continuity of the given function (see, e.g., \cite[p. 153]{Bartle1976book}, \cite[pp. 62--63]{Fitzpatrick2006book}, \cite[pp. 282--283]{JacobEvans2016book}, \cite[pp. 57, 62]{Mercer2014book}, \cite[pp. 258--259]{Sagan1974book}, \cite[pp. 238--239]{Tao2016bookI}). 

\subsection{Being an intermediate: this is a boundary value problem}\label{subsec:Intermediate}
Is it possible to formulate an Intermediate Value Theorem which not only generalizes Theorem \bref{thm:ClassicInter} but also  allows a class of discontinuous functions? As we show below (Theorem \bref{thm:NonClassicInter}), the answer is positive once we interpret the meaning of the intermediate value property as follows: if $f:X\to Y$ passes through both a subset $D\subseteq Y$ and through its complement $Y\backslash D$, then $f$ must pass through the boundary $\partial D$, which can be thought of as being an intermediate set between $D$ and $Y\backslash D$ (or between the interior $\Int(D)$ of $D$ and its exterior $\Ext(D):=X\backslash (D\cup \partial D)$). 

Before formulating the theorem, we need to recall some terminology and notation. A topological space $X$ is called connected if it cannot be represented as $X=A\cup B$, where $A$ and $B$ are two nonempty, disjoint and open subsets in $X$, or equivalently, two nonempty, disjoint and closed subsets of $X$. As is well known, every interval contained in  $\mathbb{R}$ is a connected space. 

\begin{theorem}\label{thm:NonClassicInter}
Let $(X,\mathcal{T}_X)$ be a connected topological space and let $(Y,\mathcal{T}_Y)$ be a
topological space. Suppose that $D\subseteq Y$ and $f:X\to Y$ are given.
If there are $a,b\in X$ such that $f(a)\in D$ and $f(b)\notin D$ and if either both the inverse images $f^{-1}(\Int(D))$ and $f^{-1}(\Ext(D))$ are open or both of these subsets are closed, then there exists $x\in X$ such that $f(x)\in \partial D$. In particular, if $f$ is continuous on $X$ and there are $a,b\in X$ such that $f(a)\in D$ and $f(b)\notin D$, then there exists $x\in X$ such that $f(x)\in \partial D$.
\end{theorem}
\begin{proof}
Assume first that both $f^{-1}(\Int(D))$ and $f^{-1}(\Ext(D))$ are open. The proof in the case where both of these subsets are closed is similar. If $f(a)\in \partial D$ or $f(b)\in \partial D$, then the proof is complete. Otherwise, we have $f(a)\in D\backslash \partial D$ and $f(b)\notin D\cup \partial D$. Thus $f(a)\in \Int(D)$ and $f(b)\in \Ext(D)$, and so $a\in f^{-1}(\Int(D))$ and
$b\in f^{-1}(\Ext(D))$. Hence $f^{-1}(\Int(D))$ and $f^{-1}(\Ext(D))$ are nonempty sets which are also open by our assumption. Now, since
\begin{multline*}
X=f^{-1}(Y)=f^{-1}(\Int(D)\cup \partial D \cup
\Ext(D))=f^{-1}(\Int(D))\cup f^{-1}(\partial D)\cup
f^{-1}(\Ext(D)),
\end{multline*}
it follows that if $f^{-1}(\partial D)$ is empty, then $X$ is a union
of two open, disjoint and nonempty sets, and this contradicts the
assumption that $X$ is connected. Hence $f^{-1}(\partial D)$ is
nonempty, that is, there exists $x\in X$ such that $f(x)\in \partial D$, as required. Finally, assume that $f$ is continuous on $X$ and there are $a,b\in X$ such that $f(a)\in D$ and $f(b)\notin D$. The continuity assumption on $f$ implies that both $f^{-1}(\Int(D))$ and $f^{-1}(\Ext(D))$ are open. Hence, by what we proved above there exists $x\in X$ such that $f(x)\in \partial D$, as claimed.
\end{proof}
One can think of the assumption that both $f^{-1}(\Int(D))$ and $f^{-1}(\Ext(D))$ are open as expressing a weak form of continuity, and to say that $f$ is inverse-open with respect to both $\Int(D)$ and $\Ext(D)$ (or inverse-closed with respect to these sets if both $f^{-1}(\Int(D))$ and $f^{-1}(\Ext(D))$ are closed). The following example shows that this kind of continuity is indeed very weak. 
\begin{example}
Let $X:=\R$, $Y:=\R$ and $f:X\to Y$ be defined  by $f(x):=x$ when $x$ is irrational, $f(x):=2x$ when $x\in\Q\backslash \{1/n: n\in \N\}$ and $f(1/n):=1$ whenever $n\in \N$. Let $D:=(0,\infty)$. Then $D=\Int(D)$, $\Ext(D)=(-\infty,0)$, $f(1)\in D$  and $f(-1)\notin D$. Moreover, $f^{-1}(\Int(D))=(0,\infty)$ and  $f^{-1}(\Ext(D))=(-\infty,0)$. As a result, the conditions of Theorem \bref{thm:NonClassicInter} are satisfied, and indeed  $f(0)\in \partial D=\{0\}$. But $f$  is discontinuous at every point. 
\end{example}
There is another theorem which generalizes the Intermediate Value Theorem. It says that the image of a connected topological space by a continuous function is a connected topological space \cite[Theorem 23.5, p. 150]{Munkres2000book},\cite[p. 93]{Rudin1976book}. There are two main
differences between this theorem and Theorem \bref{thm:NonClassicInter}. First, in this theorem the intermediate value property is expressed in the connectivity of $f(X)$, while in Theorem \bref{thm:NonClassicInter} it is expressed in the manner mentioned in the beginning of this subsection. Second, this theorem assumes that $f$ is continuous, while Theorem \bref{thm:NonClassicInter} allows $f$ to be discontinuous.

\subsection{Down-to-Earth + abstract space: the next generation}\label{subsec:DownToEarth2}
A simple down-to-earth application of Theorem \bref{thm:NonClassicInter} is to prove Theorem  \bref{thm:ClassicInter}, as done below.
\begin{proof}[Proof of Theorem \bref{thm:ClassicInter}] The assertion is obviously satisfied if $f(a)=\gamma$ or $f(b)=\gamma$. From now on assume that $\gamma\notin\{f(a),f(b)\}$. Assume first that $f(a)<\gamma<f(b)$ and denote $D:=(-\infty,\gamma)$. Then 
$f(a)\in \Int(D)=D$, $f(b)\in \Ext(D)=(\gamma,\infty)$ and
$f^{-1}(\Int(D))$ and $f^{-1}(\Ext(D))$ are open because $f$ is continuous.
Since $I$ is connected,  by Theorem \bref{thm:NonClassicInter} there
is $x\in I$ such that $f(x)\in
\partial D=\{\gamma\}$, i.e., $f(x)=\gamma$. The proof in the case where $f(b)<\gamma<f(a)$ follows a similar reasoning, where now we re-define $D:=(\gamma,\infty)$. 
\end{proof}

 Another down-to-earth and somewhat unexpected application of Theorem \bref{thm:NonClassicInter} is the possibility to approximate, to any desired precision, an intermediate point, namely of  a point  $x\in X$ for which $f(x)\in \partial D$. At first glance this task seems to be impossible, since the proof of Theorem \bref{thm:NonClassicInter} is a pure existence proof, that is, a proof without any single constructive clue. Despite this, \emph{sometimes the above-mentioned task can be realized}. For example, assume that $X:=[a,b]\subset \mathbb{R}$ and that the conditions needed in Theorem \bref{thm:NonClassicInter} hold (in particular, $f(a)\in D$ and $f(b)\notin D$). Denote $a_0:=a$, $b_0:=b$ and $f_0:=f$. Theorem ~\bref{thm:NonClassicInter} ensures that there exists $x_0\in [a_0,b_0]$ such that $f_0(x_0)\in \partial D$. Consider the  point $P_1:=\frac{1}{2}(a_0+b_0)$. Either $f_0(P_1)\in D$ or $f_0(P_1)\notin D$. In the first case let $a_1:=P_1$ and $b_1:=b_0$, and in the second case let $a_1:=a_0$ and $b_1:=P_1$. Denote by $f_1$ the restriction of $f_0$ to $[a_1,b_1]$. We have $f_1(a_1)\in D$, $f_1(b_1)\notin D$, $[a_1,b_1]\subset X$ and $|b_1-a_1|=0.5|b_0-a_0|$. Moreover, our assumption on $f_0^{-1}(\Int(D))$ and $f_0^{-1}(\Ext(D))$ implies that either both $f_1^{-1}(\Int(D))$ and $f_1^{-1}(\Ext(D))$ are open in $[a_1,b_1]$ or both of them are closed there. Hence Theorem ~\bref{thm:NonClassicInter} implies that there exists $x_1\in [a_1,b_1]$ such that $x_1\in \partial D$. Since we know $a_1$ and $b_1$ explicitly and since the length of $[a_1,b_1]$ is half of the length of $[a_0,b_0]$, this means that we have a better estimate for $x_1$ than the estimate that we had $x_0$. By repeating this process one essentially obtains the bisection method (but in a non-standard setting in which the function is not necessarily continuous) and finds an approximate intermediate point which deviates, in the $n$-th step, from a true intermediate point by at most $(b-a)\cdot 2^{-n}$.

It is also possible to use Theorem \bref{thm:NonClassicInter} in order to prove a somewhat abstract space version  of the classical Intermediate Value Theorem, namely \cite[Theorem 24.3, p. 154]{Munkres2000book} in which connected linearly ordered topological spaces (Definition \bref{def:Order}) appear. 
\begin{theorem}\label{thm:InterOrder}
Let $(X,\mathcal{T}_X)$ be a connected topological space and let $(Y,\leq,\mathcal{T}_L)$ be a linearly ordered topological space. Assume that $f:X\to Y$ is continuous. Given $a,b\in X$, if $\gamma\in Y$ lies between $f(a)$ and $f(b)$, then there exists $x\in X$ such that $f(x)=\gamma$.
\end{theorem}
The proof is similar to the proof of Theorem \bref{thm:ClassicInter},  where now we define $D:=I_{<\gamma}$ if $f(a)<\gamma<f(b)$ and $D:=I_{>\gamma}$ if $f(b)<\gamma<f(a)$, and we observe that $\partial D\subseteq \{\gamma\}$.

\section{Uniform continuity: the final frontier}\label{sec:UniformContinuity}
A well-known theorem, which is sometimes called the ``Uniform Continuity Theorem'' or the  ``Heine-Cantor Theorem'', says that any real continuous function defined on a closed and bounded interval $X$ of $\R$ is uniformly continuous, i.e., for every $\epsilon>0$ there exists $\delta>0$ such that for all $x,y\in X$ satisfying  $|x-y|<\delta$, we have  $|f(x)-f(y)|<\epsilon$. A more general version of this theorem says that  a continuous function $f:X\to Y$ acting between a compact metric space $(X,d_X)$ and a metric space $(Y,d_Y)$ is uniformly continuous, namely for each $\epsilon>0$ there exists $\delta>0$ such that for all $x\in X$ and $y\in X$ satisfying  $d_X(x,y)<\delta$, we have  $d_Y(f(x),f(y))<\epsilon$. Familiar proofs of this theorem, for instance, the ones which appear in 
\cite[p. 229]{Folland1984book}, \cite[pp. 87-88]{HewittStromberg1965book},
\cite[pp. 273-274]{Hille1966book},  \cite[pp. 19-20]{Jones2001book}, \cite[p. 193]{Kitchen1968book}, \cite[pp. 33-34]{Lang1969book}, \cite[p. 395]{Mercer2014book}, \cite[p. 168-169]{Nitecki2009book}, \cite[pp. 48-49, 157]{Royden1988book}, \cite[p. 91]{Rudin1976book},  \cite[p. 114]{Sagan1974book}, \cite[p. 143-144]{Spivak1994book}, \cite[pp. 247-248]{Tao2016bookI}, and \cite[pp. 323-324, 682]{ThomsonBrucknerBruckner2008book}, show the existence of such a positive number $\delta$, but they do not explain how to find it explicitly. In particular, no information is provided regarding how to find the largest possible such $\delta$ (the optimal delta).

\subsection{The optimal delta: a new hope}\label{subsec:NewHope}
Is it possible to find explicitly the optimal $\delta$? Proposition \bref{prop:MaximalDelta} below shows that the answer is positive. A key step in establishing this proposition is simply to reformulate the condition of uniform continuity, as done in Lemma \bref{lem:EquivalentConditionUniCont} below. The uniform continuity of a continuous function defined on a compact metric space, as well as more general results, follow  as a consequence (Theorem \bref{thm:UniCont} below). 

\begin{lemma}\label{lem:EquivalentConditionUniCont} 
Let $(X,d_X)$ and $(Y,d_Y)$ be two metric spaces, and let $f:X\to Y$. 
Then $f$ is uniformly continuous  if and only if for each $\epsilon>0$ there exists $\delta>0$  such that for all $x,y\in X$ satisfying $d_Y(f(x),f(y))\geq\epsilon$, we have $d_X(x,y)\geq \delta$. 
\end{lemma}
\begin{proof}
The assertion follows directly from the definitions, using contrapositive (any $\epsilon>0$ and $\delta>0$ which satisfy the first condition are good for the second one, and vice versa). 
\end{proof}
In  other words, $f$ is uniformly continuous  if and only if for each $\epsilon>0$ there exists $\delta>0$  such that for all $(x,y)\in A_{\epsilon}$ we have $d_X(x,y)\geq \delta$, where 
\begin{equation}\label{eq:A_epsilon}
A_{\epsilon}:=\{(x,y)\in X^2: d_Y(f(x),f(y))\geq \epsilon\}.
\end{equation}
An obvious property of $A_{\epsilon}$ is that $(x,x)\notin A_{\epsilon}$ for all $\epsilon>0$ and $x\in X$.

The following proposition introduces the optimal delta and describes some properties of it.
\begin{proposition}\label{prop:MaximalDelta}
Let $(X,d_X)$ and $(Y,d_Y)$ be two metric spaces. Given $f:X\to Y$, 
let $\delta_f:[0,\infty) \to [0,\infty]$ be defined by
\begin{equation}\label{eq:delta_f}
\delta_f(\epsilon):=\left\{\begin{array}{lll}
\inf \{d_X(x,y): (x,y)\in A_{\epsilon}\}, & \textnormal{if}\,\epsilon\in [0,\infty)\,\textnormal{and}\, A_{\epsilon}\neq \emptyset\\
\infty, & \textnormal{if}\,\epsilon\in [0,\infty)\,\textnormal{and}\,\,A_{\epsilon}=\emptyset,
\end{array}
\right.
\end{equation}
where $A_{\epsilon}$ is defined in \beqref{eq:A_epsilon}. Then the following properties hold:
\begin{enumerate}[(i)]
\item\label{item:delta_f is increasing and finite} $\delta_f$ is nonnegative, monotone increasing, and satisfies $\delta_f(0)=0$. In addition,  given $\epsilon\in [0,\infty)$, we have that  $\delta_f(\epsilon)$ is  finite if and only if $A_{\epsilon}\neq \emptyset$. In particular, $\delta_f$ is finite on the set  $\{0\}\cup [0,M_f)$, where $M_f$ is the oscillation of $f$, namely 
\begin{equation}\label{eq:M_f}
M_f:=\sup\{d_Y(f(x),f(y)): x,y\in X\}. 
\end{equation}
Finally, when $M_f<\infty$, then $\delta_f$ is infinite on $(M_f,\infty)$.
\item\label{item:delta_f is maximal} If $\delta_f(\epsilon)>0$ for each $\epsilon\in (0,\infty)$, then $f$  is uniformly continuous; moreover, $\delta_f$ assigns to each $\epsilon>0$  the optimal delta, that is, the largest possible delta from the definition of uniform continuity (in particular, when $\delta_f(\epsilon)=\infty$, then any $\delta\in (0,\infty)$ can be associated with $\epsilon$ in this definition). If $\delta_f(\epsilon)=0$ for some $\epsilon>0$, then $f$ is not uniformly continuous. In particular, $f$ is uniformly continuous if and only if $\delta_f(\epsilon)>0$ for each $\epsilon\in (0,\infty)$. 
\end{enumerate}
\end{proposition}

\begin{proof}
\begin{enumerate}[(i)]
\item The assertions are a simple consequence of \beqref{eq:A_epsilon},\beqref{eq:delta_f}, and \beqref{eq:M_f} (for instance, consider the assertion regarding $\{0\}\cup[0,M_f)$: if $M_f>0$, then for each $\epsilon\in [0,M_f)$ there exists, by the definition of $M_f$, a pair $(x,y)\in X^2$ satisfying $\epsilon<d_Y(f(x),f(y))$; thus $(x,y)\in A_{\epsilon}$ and hence \beqref{eq:delta_f} implies that $\delta_f(\epsilon)$ belongs to the interval $[0,d(x,y)]$, namely it is finite). 

\item Suppose that $\delta_f(\epsilon)>0$ for all $\epsilon\in (0,\infty)$. Fix arbitrary  $\epsilon\in (0,\infty)$ and $\delta\in (0,\delta_f(\epsilon))$. Given $x,y\in X$ satisfying $d_X(x,y)<\delta$, we must have $d_Y(f(x),f(y))<\epsilon$. Indeed, suppose to the contrary that this inequality is violated; then $(x,y)\in A_{\epsilon}$ by \beqref{eq:A_epsilon} and hence, from \beqref{eq:delta_f}, we have $\delta_f(\epsilon)\leq d_X(x,y)$, a contradiction because $d_X(x,y)<\delta<\delta_f(\epsilon)$ by our assumptions. We conclude that the assumption $\delta_f(\epsilon)>0$ for all $\epsilon\in (0,\infty)$ implies that $f$ is uniformly continuous. 

Now fix some $x,y\in X$ satisfying $d_X(x,y)<\delta_f(\epsilon)$. It must be that $d_Y(f(x),f(y))<\epsilon$, because if not, then we have $d_Y(f(x),f(y))\geq \epsilon$ and therefore $(x,y)\in A_{\epsilon}$; thus \beqref{eq:delta_f} implies that $\delta_f(\epsilon)\leq d_X(x,y)$, a contradiction. Thus if $\delta_f(\epsilon)$ is finite, then it can be used as a delta associated with $\epsilon$ in the  definition of uniform continuity.  Moreover, if  $\delta<\delta_f(\epsilon)$ and $d_X(x,y)<\delta$ for some $(x,y)\in X^2$, then $d_X(x,y)<\delta_f(\epsilon)$, and hence from the previous lines we conclude that $d_Y(f(x),f(y))<\epsilon$. Thus any $\delta\in (0,\delta_f(\epsilon))$ can be associated with $\epsilon$ in the definition of uniform continuity of $f$. 

In order to show that $\delta_f(\epsilon)$ is the largest possible delta associated with $\epsilon$ in the  definition of uniform continuity, we still need to show that any other $\delta>0$ associated with $\epsilon$ is not greater than $\delta_f(\epsilon)$. Indeed, if $A_{\epsilon}=\emptyset$, then  \beqref{eq:delta_f} implies that $\delta<\delta_f(\epsilon)=\infty$, as required.  Suppose now that $A_{\epsilon}\neq \emptyset$ and let $(x,y)\in A_{\epsilon}$. It must be that  $\delta\leq d_X(x,y)$, because if this inequality is not true, then the choice of $\delta$ and the fact that $f$ is uniformly continuous imply that $d_Y(f(x),f(y))<\epsilon$, a contradiction to the assumption that $(x,y)\in A_{\epsilon}$. We conclude that  $\delta$ is a lower bound of the set $\{d_X(x,y):(x,y)\in A_{\epsilon}\}$. Because $\delta_f(\epsilon)$ is the maximal such a lower bound as follows from \beqref{eq:delta_f}, it follows that $\delta\leq \delta_f(\epsilon)$. To conclude, $\delta\leq\delta_f(\epsilon)$ no matter if $A_{\epsilon}=\emptyset$ or $A_{\epsilon}\neq \emptyset$, and hence $\delta_f(\epsilon)$ is indeed the optimal delta.

Now suppose that $\delta_f(\epsilon)=0$ for some $\epsilon>0$. Assume to the contrary that $f$ is uniformly continuous. Since $\delta_f(\epsilon)$ is finite, we have $A_{\epsilon}\neq\emptyset$ (by (Part \beqref{item:delta_f is increasing and finite}). By Lemma  \bref{lem:EquivalentConditionUniCont} there exists $\delta>0$ such that $d_X(x,y)\geq \delta$ for all $(x,y)\in A_{\epsilon}$. Thus $\delta$ is a positive lower bound of the set $\{d_X(x,y):(x,y)\in A_{\epsilon}\}$. Since $\delta_f(\epsilon)$ is the maximal such a lower bound, we have $\delta\leq \delta_f(\epsilon)=0$, a contradiction. Thus $f$ is not uniformly continuous. Finally, from previous lines we see that $f$ is uniformly continuous if and only if $\delta_f(\epsilon)>0$ for each $\epsilon\in (0,\infty)$. 
\end{enumerate}
\end{proof}

The optimal delta  modulus $\delta_f$  from \beqref{eq:delta_f} can be  thought of as being a modulus which is dual to to the modulus of (uniform) continuity 
\begin{equation}\label{eq:w}
w_f(\delta):=\sup\{d_Y(f(x),f(y)): x,y\in X, \,d_X(x,y)\leq \delta\}.
\end{equation} 
 Local versions of $\delta_f$ can be defined too, i.e.,  
\begin{equation*}
\delta_f(\epsilon,x):=\inf \{d_X(x,y):  y\in X,\,d_Y(f(x),f(y))\geq \epsilon\},\quad \forall \epsilon\in [0,\infty),\,\forall x\in X,
\end{equation*}
where $\inf\emptyset:=\infty$. In other words (and using a reasoning similar to the proof of Proposition \bref{prop:MaximalDelta}), if we fix some point $x\in X$ and a positive number $\epsilon$, then $\delta_f(\epsilon,x)$ describes  the optimal delta associated with $\epsilon$ and $x$ in the definition of continuity of $f$ at $x$. 

Interestingly, the setting needed for the definition of $\delta_f$ is wider than metric spaces, since in Lemma \bref{lem:EquivalentConditionUniCont}  and Proposition \bref{prop:MaximalDelta} not all of the assumptions in the definition of metric spaces have been used (for example, neither the triangle inequality nor the symmetry of the distance function have been used). Thus $\delta_f$ may be useful for distance functions, divergences and distortion measures used in data processing \cite{Basseville2013jour}, data analysis \cite{Kogan2007book}, information theory \cite{Gray2013book} and many other scientific and technological areas \cite{DezaDeza2016book}. 

\subsection{The compactness strikes back}\label{subsec:Compactness}
Using tools developed earlier, we can now prove a general version of the Uniform Continuity Theorem, a version in which the \emph{a priori}  condition on the involved function is weaker than continuity, and the \emph{a priori} condition on the involved space is weaker than compactness. 

\begin{theorem}\label{thm:UniCont}
Let $(X,d_X)$ and $(Y,d_Y)$ be two metric spaces and let $f:X\to Y$. Consider the following statements:
\begin{enumerate}[(i)]
\item\label{item:f} $(X,d_X)$ is compact and $f$ is continuous;
\item\label{item:F_epsilon} $(X,d_X)$ is compact and the function $F:X^2\to \R$ which is defined by $F(x,y):=d_Y(f(x),f(y))$ for each $(x,y)\in X^2$ is upper semicontinuous (Definition \bref{def:SemiContinuity});
\item\label{item:A_epsilonClosed} $(X,d_X)$ is compact and $A_{\epsilon}$ (from \beqref{eq:A_epsilon}) is closed in $X^2$ for all $\epsilon>0$;
\item\label{A_epsilonIsCompact}  $A_{\epsilon}$  is compact for each $\epsilon>0$;
\item\label{item:either_epsilon} for each $\epsilon>0$, either $A_{\epsilon}=\emptyset$ or  $A_{\epsilon}\neq \emptyset$ and the function $d_X$ attains a  minimum on $A_{\epsilon}$ at some point $(x_0,y_0)\in A_{\epsilon}$; moreover, in the second case $\delta_f(\epsilon)=d_X(x_0,y_0)$; 
\item\label{item:f_is_UniCont} $f$ is uniformly continuous.
\end{enumerate}
Then \beqref{item:f} $\Lra$ \beqref{item:F_epsilon} $\Lra$ \beqref{item:A_epsilonClosed} $\Lra$ \beqref{A_epsilonIsCompact} $\Lra$ \beqref{item:either_epsilon} $\Lra$ \beqref{item:f_is_UniCont}. 
\end{theorem}
\begin{proof}
\beqref{item:f} $\Lra$ \beqref{item:F_epsilon}: In this case $F$ is even continuous as follows from the triangle inequality and the continuity of $f$.

\beqref{item:F_epsilon} $\Lra$ \beqref{item:A_epsilonClosed}: Let $\epsilon>0$ be arbitrary. From \beqref{eq:A_epsilon} the set $A_{\epsilon}$ is nothing but the $\geq$-level-set $F^{\geq \epsilon}$ (Definition \bref{def:SemiContinuity}) and hence it is closed in $X^2$ since we assume that $F$ is upper semicontinuous.

\beqref{item:A_epsilonClosed} $\Lra$ \beqref{A_epsilonIsCompact}: Let $\epsilon>0$ be arbitrary. Since $(X,d_X)$ is compact, also $(X^2,d_{X^2})$ is compact, with, say,
\begin{equation*}
d_{X^2}((x_1,x_2),(y_1,y_2)):=\sqrt{(d_X(x_1,y_1))^2+(d_X(x_2,y_2))^2}\,\,\,\,\forall (x_1,x_2),(y_1,y_2)\in X^2. 
\end{equation*}
Because we assume that $A_{\epsilon}\subseteq X^2$ is closed, it follows that $A_{\epsilon}$ is compact as a closed subset of a compact space. 

\beqref{A_epsilonIsCompact} $\Lra$ \beqref{item:either_epsilon}: Let $\epsilon>0$ be arbitrary. If $A_{\epsilon}=\emptyset$, then there is nothing to prove. Assume now that $A_{\epsilon}\neq\emptyset$. The restriction of $d_X$ to $A_{\epsilon}$ is a real-valued continuous function defined on the compact space $A_{\epsilon}$. Hence the Extreme Value Theorem (Theorem  \bref{thm:Weierstrass} or Theorem ~\bref{thm:ExtremeValueSemicontunuity}) implies that $d_X$ has a minimizer $(x_0,y_0)$ in $A_{\epsilon}$. 
It follows from \beqref{eq:delta_f} that   $\delta_f(\epsilon)=d_X(x_0,y_0)$, as required. 

\beqref{item:either_epsilon} $\Lra$ \beqref{item:f_is_UniCont}
According to Proposition \bref{prop:MaximalDelta}\beqref{item:delta_f is maximal}, for proving that $f$ is uniform continuous on $X$ it is sufficient to show that the optimal delta from \beqref{eq:delta_f}  satisfies $\delta_f(\epsilon)>0$ for each $\epsilon>0$. Let $\epsilon\in (0,\infty)$ be given.  If $A_{\epsilon}=\emptyset$, then $\delta_f(\epsilon)=\infty>0$, as asserted. Assume now that $A_{\epsilon}\neq\emptyset$. By our assumption there exists a minimizer $(x_0,y_0)$ of $d_X$ on  $A_{\epsilon}$. From \beqref{eq:A_epsilon} we have $x_0\neq y_0$, and from \beqref{eq:delta_f} we have $\delta_f(\epsilon)=d_X(x_0,y_0)>0$, as required.
\end{proof}

\begin{example}\label{ex:delta_f}
Let $\alpha\in (0,\infty)$ and $b\in (0,\infty]$ be fixed. Define $X:=[0,b]$ if $b<\infty$ and $X:=[0,\infty)$ if $b=\infty$. Let  $Y:=[0,\infty)$. Let $d_X$ be the usual absolute value metric on $X$, namely $d_X(x,y):=|x-y|$ for all $x,y\in X$. Similarly, let $d_Y$ be the absolute value metric on $Y$. Define $f:X\to Y$ by $f(x):=x^{\alpha}$ for each $x\in X$. By using the method suggested in Theorem \bref{thm:UniCont}, namely by trying to minimize the continuous function $d_X$ over $A_{\epsilon}$, and by using elementary calculus and separating into cases, one can obtain explicitly $\delta_f$ as follows (the analysis is simple, though a bit technical; it can be found in the appendix below):

\begin{equation*}
\delta_f(\epsilon)=\left\{
\begin{array}{lll}
 {b-(b^{\alpha}-\epsilon)}^{1/\alpha} & \textnormal{if}\,\,\alpha\geq 1,\,b\in (0,\infty),\,\epsilon\in [0,b^{\alpha}],\\
 0 & \textnormal{if}\,\,\alpha>1,\,b=\infty,\,\epsilon\in [0,\infty),\\
\epsilon^{1/\alpha} & \textnormal{if}\,\,0<\alpha\leq 1,\,b\in (0,\infty),\,\epsilon\in[0,b^{\alpha}],\\
\epsilon^{1/\alpha} & \textnormal{if}\,\, 0<\alpha\leq 1,\,b=\infty, \,\epsilon\in [0,\infty),\\
\infty &  \textnormal{if}\,\,\alpha\in (0,\infty),\,b\in (0,\infty),\,\epsilon>b^{\alpha}.
\end{array}\right.
\end{equation*}
\end{example}

\subsection{Return of the semicontinuity}\label{subsec:Semicontinuity}
It is tempting to conjecture, and the above example supports this conjecture, that the optimal delta $\delta_f$ is a continuous function of its variable $\epsilon$.  Unfortunately, in general this is not true. Indeed,  the following example presents a continuous function $f$ for which $\delta_f$ is discontinuous at infinitely many points  (see the appendix below for the simple, though a bit lengthy, explanation): 
\begin{example}\label{ex:delta_f Discontinuous}
Let $f:[0,1]\to \R$ be  the ``decreasing chainsaw'' function defined by $f(0):=0$, $f(1):=1$, and for all other $t\in [0,1]$ by 
\begin{equation}\label{eq:Chainsaw}
f(t):=\left\{\begin{array}{ll}
\displaystyle{\frac{1}{n}-(2n-1)\left(t-\frac{1}{n}\right)},& \textnormal{if}\,\, t\in \left[\displaystyle{\frac{1}{n}},\displaystyle{\frac{2}{2n-1}}\right],\, 2\leq n\in\N,\\\\
 \displaystyle{(2n-1)\left(t-\frac{2}{2n-1}\right)},  & \textnormal{if}\,\,t\in \left[\displaystyle{\frac{2}{2n-1}},\displaystyle{\frac{1}{n-1}}\right],\, 2\leq n\in\N.
\end{array}\right.
\end{equation}
\end{example} 
 
In light of Example \bref{ex:delta_f Discontinuous}, one may wonder if something can be done in order to save the day regarding $\delta_f$. The answer is that a few such (imperfect) possibilities  exist. The first is to abandon $\delta_f$ and instead to try to find other deltas corresponding to $\epsilon$ from the definition of uniform continuity, hopefully deltas which are  continuous as a function of $\epsilon$. As can be seen in 
\cite{AlmiraPassot2008jour, ArticoMarcon1993inproc, deMarco2001jour, Enayat2000jour, Guthrie1983jour, SeidmanChildress1975jour} and \cite[pp. 240-241]{RepovsSemenov1998book}, it turns out that in various settings it is indeed  possible to select, among the possible deltas coming from the definition of continuity, a one which is a continuous function of $\epsilon$ (and, sometimes, also of $x$). The second possibility is to continue with $\delta_f$, but to focus the attention on other properties of it with the hope that some of them are nice. This is done in the next proposition which also finishes our trek. 

\begin{proposition}
Let $(X,d_X)$ and $(Y,d_Y)$ be two metric spaces and let $f:X\to Y$. Then $\delta_f$ from \beqref{eq:delta_f} has at most countably many points of discontinuity and it is differentiable almost everywhere on $(0,M_f)$, where $M_f$ is defined in \beqref{eq:M_f}. Moreover, if $X$ is compact and $f$ is continuous, then $\delta_f$ is lower semicontinuous. 
\end{proposition}
\begin{proof}
Proposition \bref{prop:MaximalDelta}\beqref{item:delta_f is increasing and finite} ensures that  $\delta_f$ is increasing and finite on $(0,M_f)$. Thus, a theorem of Lebesgue \cite[p. 514]{Jones2001book} ensures that $\delta_f$  is differentiable almost everywhere in this interval. Since $\delta_f$ is increasing, it has at most countably many points of discontinuity \cite[p. 146]{Bartle1976book}. 

Assume now that $f$ is continuous and $X$ is compact. According to Definition  ~\bref{def:SemiContinuity}, for proving that $\delta_f$ is lower semicontinuous we need to show that for all $\alpha\in\R$ the level-set $L_{\alpha}:=\{\epsilon\in [0,\infty): \delta_f(\epsilon)\leq \alpha\}$ is closed. If  $\alpha<0$, then $L_{\alpha}=\emptyset$ and hence it is closed. Now assume to the contrary that $L_{\alpha}$ is not closed for some $\alpha\geq 0$. Then we can find a sequence $(\epsilon_n)_{n=1}^{\infty}$ of elements of $L_{\alpha}$ which converges to a nonnegative  number $\epsilon\notin L_{\alpha}$ (as a matter of fact, $\epsilon$  must be positive because $0\in L_{\alpha}$ by Proposition \bref{prop:MaximalDelta}\beqref{item:delta_f is increasing and finite}). Therefore  $\delta_f(\epsilon)>\alpha$ and hence we can choose some $q\in (\alpha,\delta_f(\epsilon))$. From \beqref{eq:delta_f} and the fact that $\delta_f(\epsilon_n)\leq\alpha<q$ for all $n\in\N$, there exists, for each $n\in\N$, a pair $(x_n,y_n)\in A_{\epsilon_n}$ satisfying $d_X(x_n,y_n)<q$ and $d_Y(f(x_n),f(y_n))\geq \epsilon_n$. Since $X$ is compact, we can find a subsequence $(x_{n_{k}})_{k=1}^{\infty}$ of $(x_n)_{n=1}^{\infty}$ which converges to some $x\in X$ and a subsequence $(y_{n_{k_j}})_{j=1}^{\infty}$ of $(y_{n_k})_{k=1}^{\infty}$ which converges to some $y\in X$. 

We have  $\lim_{j\to\infty}d_Y(f(x_{n_{k_j}}),f(y_{n_{k_j}}))=d_Y(f(x),f(y))$ since both $f$ and $d_Y$ are continuous. On the other hand, from the inequality $d_Y(f(x_n),f(y_n))\geq \epsilon_n$ (which holds, in particular, for $n_{k_j}$ for each $j\in\N$) and the fact that $\lim_{n\to\infty}\epsilon_n=\epsilon$ it follows that $d_Y(f(x),f(y))\geq \epsilon$. Thus $(x,y)\in A_{\epsilon}$ and hence $\delta_f(\epsilon)\leq d_X(x,y)$. But we already know that $d_X(x_n,y_n)<q$ for each $n\in \N$. Consequently, by taking $n$ to be 
$n_{k_j}$, letting $j\to\infty$, and using the  continuity of $d_X$, we have $d_X(x,y)\leq q$. Now we combine this inequality with the inequality $\delta_f(\epsilon)\leq d_X(x,y)$ and the choice $q\in (\alpha,\delta_f(\epsilon))$, and observe that we arrived at the impossible inequality 
$\delta_f(\epsilon)\leq d_X(x,y)\leq q<\delta_f(\epsilon)$. 
This contradiction shows that $L_{\alpha}$ is closed and $\delta_f$ is lower  semicontinuous, as required. 
\end{proof}

\section*{Appendix}
In this appendix we present the full analysis related to Example \bref{ex:delta_f} and Example \bref{ex:delta_f Discontinuous}.
\begin{proof}[{\bf Full analysis of Example \bref{ex:delta_f}}]
Let $\epsilon\geq 0$ be fixed. Following Theorem \bref{thm:UniCont}, in order to compute $\delta_f(\epsilon)$ it is useful to investigate the function $d_X$ on $A_{\epsilon}:=\{(x,y)\in X^2: |f(x)-f(y)|\geq \epsilon\}$ and to find its minimizers there (if there are any). Proposition \bref{prop:MaximalDelta}\beqref{item:delta_f is increasing and finite} ensures that $\delta_f(0)=0$. Assume first that $b<\infty$. It must be that for $\epsilon>b^{\alpha}$ one has $A_{\epsilon}=\emptyset$,  because if $(x,y)\in A_{\epsilon}$, then, in particular, $x\in [0,b]$ and $y\in [0,b]$ and hence $|f(x)-f(y)|=|x^{\alpha}-y^{\alpha}|\leq b^{\alpha}-0<\epsilon$, a contradiction.  Hence $\delta_f(\epsilon)=\infty$ whenever  $\epsilon>b^{\alpha}$. 

From now on (as long as $b<\infty$) we assume that $\epsilon\in (0,b^{\alpha}]$. The set  of minimizers of $d_X$ on $A_{\epsilon}$ coincides with the set of minimizers (on $A_{\epsilon}$) of the function $h:\R^2\to\R$ defined by $h(x,y):=(x-y)^2$ for each $(x,y)\in \R^2$. Since $h$ is smooth, if a minimum of it is attained at a point $(x,y)$ in the interior of  $A_{\epsilon}$, then Fermat's principle from basic calculus implies that $(0,0)=\nabla h(x,y)=(2(x-y),2(x-y))$, and therefore $x=y$, a contradiction to \beqref{eq:A_epsilon}. Thus any minimizer of $h$, and hence of $d_X$, must be located on the boundary of  $A_{\epsilon}$ (as a subset of $\R^2$).

Since $b<\infty$, the set $A_{\epsilon}$ is compact and hence the Extreme Value Theorem guarantees that $d_X$ has at least one  minimizer on it. Moreover, since $b<\infty$, it follows that $A_{\epsilon}$ is composed of two curved triangles (boundaries+interiors) which are symmetric relative to the diagonal $\{(x,y)\in X^2: y=x\}$. Denote these triangles by $T_{\textnormal{up}}$ and $T_{\textnormal{down}}$. The boundary of $T_{\textnormal{up}}$ can be written as  $\Gamma_{1,\textnormal{up}}\cup\Gamma_{2,\textnormal{up}}\cup\Gamma_{3,\textnormal{up}}$ and the boundary of  $T_{\textnormal{down}}$ can be written as  $\Gamma_{1,\textnormal{down}}\cup\Gamma_{2,\textnormal{down}}\cup\Gamma_{3,\textnormal{down}}$, where these sets are defined as follows: $\Gamma_{1,\textnormal{up}}:=\{(x,y)\in X^2: x=0, y\in [\epsilon^{1/\alpha},b]\}$, $\Gamma_{2,\textnormal{up}}:=\{(x,y)\in X^2: x\in [0,(b^{\alpha}-\epsilon)^{1/\alpha}], y=b\}$, $\Gamma_{3,\textnormal{up}}:=\{(x,y)\in X^2: x\in [0,(b^{\alpha}-\epsilon)^{1/\alpha}], y=(x^{\alpha}+\epsilon)^{1/\alpha}\}$  and similarly with $T_{\textnormal{down}}$: see Figures \bref{fig:A_epsilon_alpha_is_3}--\bref{fig:A_epsilon_alpha_is_third}.

\begin{figure}[t]
\begin{minipage}[t]{0.49\textwidth}
\begin{center}
{\includegraphics[clip, scale=0.42]{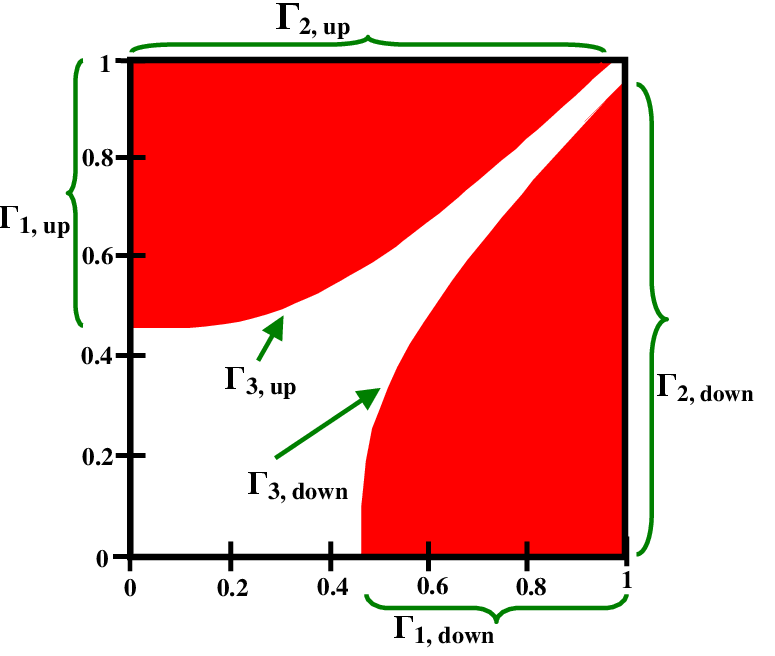}}
\end{center}
 \caption{$A_{\epsilon}$ of Example \bref{ex:delta_f} when $\alpha=3$, $\epsilon=0.1$, $b=1$.}
\label{fig:A_epsilon_alpha_is_3}
\end{minipage}
\hfill
\begin{minipage}[t]{0.495\textwidth}
\begin{center}
{\includegraphics[scale=0.42]{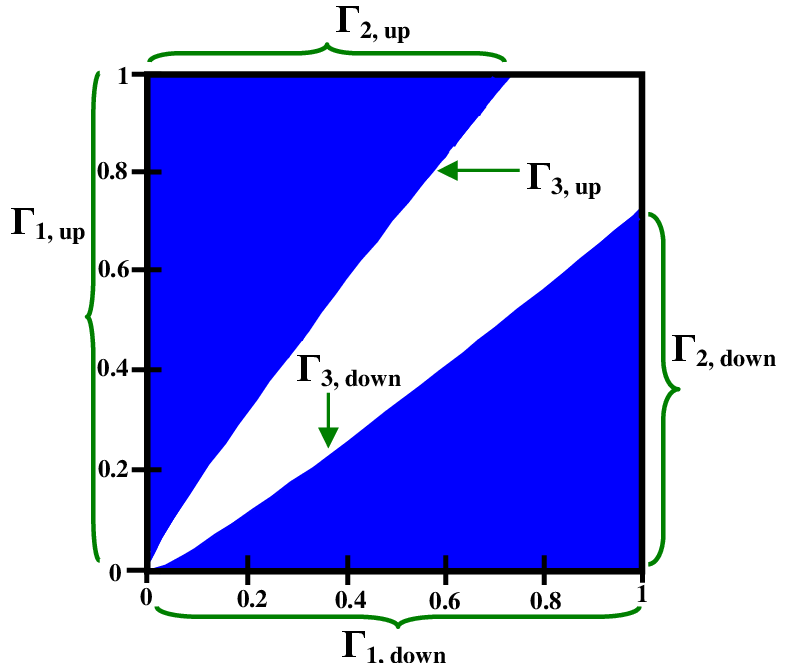}}
\end{center}
 \caption{$A_{\epsilon}$ of Example \bref{ex:delta_f} when $\alpha=1/3$, $\epsilon=0.1$, $b=1$.}
\label{fig:A_epsilon_alpha_is_third}
\end{minipage}
\end{figure}

From elementary calculus it follows that the restriction of $d_X$ to each of these curves can be written as a  one-dimensional monotone function (increasing or decreasing: depending whether $\alpha\geq 1$ or $\alpha\leq 1$) and its minimal values are attained at the corners $(0,\epsilon^{1/\alpha})$, $((b^{\alpha}-\epsilon)^{1/\alpha},b)$, $(\epsilon^{1/\alpha},0)$, $(b,(b^{\alpha}-\epsilon)^{1/\alpha})$. Hence these minimal values are either $|0-\epsilon^{1/\alpha}|$ or $|(b^{\alpha}-\epsilon)^{1/\alpha}-b|$, namely  $\epsilon^{1/\alpha}$ or $b-(b^{\alpha}-\epsilon)^{1/\alpha}$. To see which of these values is smaller, as a function of $\epsilon$ (where $\epsilon\in (0,b^{\alpha}]$), consider the function $g(\epsilon):=b-(b^{\alpha}-\epsilon)^{1/\alpha}-\epsilon^{1/\alpha}$ on the interval $(0,b^{\alpha}]$. Elementary calculus shows that $g$ is nonnegative on this interval when $\alpha\in (0,1]$, and nonpositive there when $\alpha\in [1,\infty)$. It follows that the minimal value of $d_X$ on $A_{\epsilon}$ is $\epsilon^{1/\alpha}$ if $\alpha\in (0,1]$  and it is $b-(b^{\alpha}-\epsilon)^{1/\alpha}$ when $\alpha\in [1,\infty)$. Thus Theorem \bref{thm:UniCont} implies that $\delta_f(\epsilon)=\epsilon^{1/\alpha}$ if $\alpha\in (0,1]$ and  $\delta_f(\epsilon)=b-(b^{\alpha}-\epsilon)^{1/\alpha}$ if $\alpha\in [1,\infty)$. 

Finally, we need to consider the case where $b=\infty$. If $\alpha\in (0,1]$, then it must be that $\delta_f(\epsilon)=\epsilon^{1/\alpha}$ for each $\epsilon\in (0,\infty)$. Indeed, fix  $\epsilon\in (0,\infty)$. Since $(0,\epsilon^{1/\alpha})\in A_{\epsilon}$, it follows from \beqref{eq:delta_f} that $0\leq \delta_f(\epsilon)\leq d_X(0,\epsilon^{1/\alpha})=\epsilon^{1/\alpha}$. If, to the contrary, $\delta_f(\epsilon)<\epsilon^{1/\alpha}$, then from \beqref{eq:delta_f} there exists a point $(x_1,y_1)\in A_{\epsilon}$ such that $d_X(x_1,y_1)<\epsilon^{1/\alpha}$. Let $\tilde{b}>\max\{x_1,y_1\}$. Then  $(x_1,y_1)\in [0,\tilde{b}]^2\cap A_{\epsilon}$. Consider the restriction of $A_{\epsilon}$ to the square $[0,\tilde{b}]^2$. As explained in the previous paragraphs (where now $\tilde{b}$ replaces $b$), the minimal value of $d_X$ on $A_{\epsilon}$ is $\epsilon^{1/\alpha}$ when $\alpha \in (0,1]$. Since we assume that $d_X(x_1,y_1)<\epsilon^{1/\alpha}$, we see that the value of $d_X$ at the point $(x_1,y_1)$, which belongs to $[0,\tilde{b}]^2\cap A_{\epsilon}$, is smaller than the minimal value of $d_X$ on $[0,\tilde{b}]^2\cap A_{\epsilon}$. This is a contradiction which implies the assertion. 

It remains to consider the case where $\alpha\in (1,\infty)$. We claim that in this case $\delta_f(\epsilon)=0$ for each $\epsilon\in (0,\infty)$. Indeed, fix arbitrary $\epsilon\in (0,\infty)$ and $r\in (0,\infty)$.  By using l'H\^opital's rule and the assumption that $\alpha\in (1,\infty)$ one can show that $\lim_{b\to\infty}(b-(b^{\alpha}-\epsilon)^{1/\alpha})=0$. Hence $b-(b^{\alpha}-\epsilon)^{1/\alpha}<r$ for all $b$ sufficiently large. Let $b$ be sufficiently large. Since for $(x,y):=((b^{\alpha}-\epsilon)^{1/\alpha},b)$ we have $(x,y)\in A_{\epsilon}$ and $d_X(x,y)=b-(b^{\alpha}-\epsilon)^{1/\alpha}$, it follows from \beqref{eq:delta_f} and the choice of $b$ that $\delta_f(\epsilon)\leq d_X(x,y)<r$. Because $r$ can be arbitrary small we conclude that $\delta_f(\epsilon)=0$, as claimed. 
\end{proof}

\begin{proof}[{\bf Full Analysis of Example \bref{ex:delta_f Discontinuous}}]
See Figure \bref{fig:chainsaw}. We claim that $\delta_f$ is discontinuous at each of the infinitely many points $1/n\in X:=[0,1]$, $2\leq n\in \N$. Indeed, fix a natural number $n\geq 2$ and let $\epsilon\in (0,1]$ satisfy $1/n<\epsilon$. Let $y:=2/(2n+1)$  and $x:=1/n$. Since $|f(y)-f(x)|=1/n$, we have $(x,y)\in A_{1/n}$.  Hence from \beqref{eq:delta_f}  it follows that 
\begin{equation}\label{eq:delta_f(1/n)} 
\delta_f\left(\frac{1}{n}\right)\leq d_X(x,y)=|y-x|=\frac{1}{(2n+1)n}. 
\end{equation}
On the other hand, we will see below that $\delta_f(\epsilon)>1/(n(2n-1))$. Since $1/(n(2n-1))>1/((2n+1)n)$,  it is  not possible to bridge the gap between $\delta_f(1/n)$ and $\delta_f(\epsilon)$ no matter how  close $\epsilon$ is to $1/n$. Therefore $\delta_f$ is discontinuous at the point $1/n\in X$. 

Indeed, consider $A_{\epsilon}$ from \beqref{eq:A_epsilon} and let $(x_0,y_0)\in A_{\epsilon}$ be a minimizer of $d_X$ on $A_{\epsilon}$ whose existence is guaranteed by Theorem \bref{thm:UniCont} (since $f$ is continuous and  $X$ is compact). Assume first that $x_0<y_0$; the case $y_0>x_0$ can be handled similarly, and the case $x_0=y_0$ is impossible due to \beqref{eq:A_epsilon}. It must be that $y_0>2/(2n-1)$. Indeed, if, to the contrary, we have $y_0\leq 2/(2n-1)$, then both $x_0$ and $y_0$ are located in the interval $[0,2/(2n-1)]$. But on this interval $f$ is bounded from above by $1/n$. Since $f$ is bounded from below by 0 (everywhere), it follows that $|f(x_0)-f(y_0)|\leq 1/n<\epsilon$, a contradiction to the assumption $(x_0,y_0)\in A_{\epsilon}$. Now let $k_0\in\N$ be the minimal $k\in\N$ such that $1/k<y_0$. Since $y_0\leq 1$, it follows that $k_0>1$. Since $k_0$ is the minimal natural number $k\in\N$ which satisfies $1/k<y_0$, it follows that $1/(k_0-1)\geq y_0$. But $y_0>2/(2n-1)>1/n$, and so $1/(k_0-1)>1/n$. Hence $n>k_0-1$. Since $n\in\N$, we have 
\begin{equation}\label{n>=k_0}
n\geq k_0.
\end{equation}

\begin{figure}[t]
\begin{minipage}[t]{1\textwidth}
\begin{center}
{\includegraphics[clip, scale=0.5]{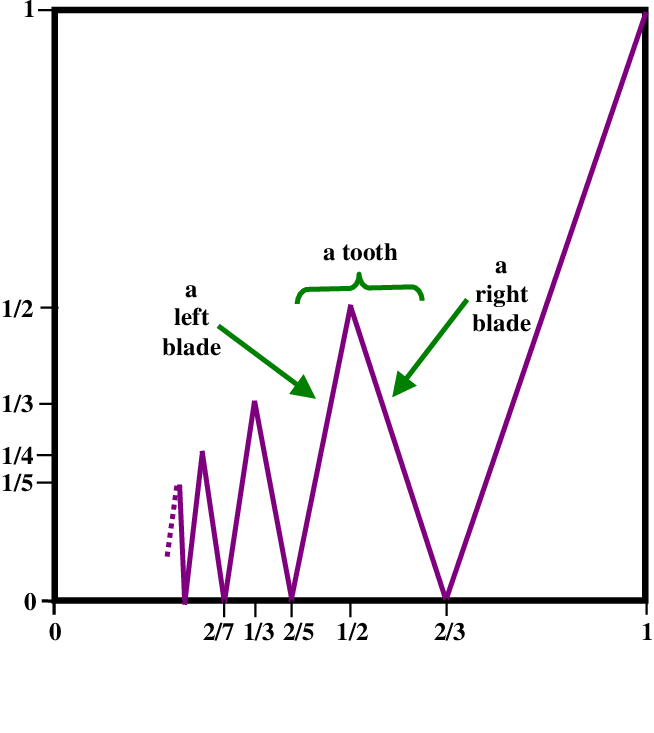}}
\end{center}
 \caption{The graph of the decreasing chainsaw function $f$ from \beqref{eq:Chainsaw}.}
\label{fig:chainsaw}
\end{minipage}
\end{figure}
 
  The rest of the analysis is done by considering several cases which can be treated in a similar manner and hence we will consider only a few of them. First, we observe (Figure \bref{fig:chainsaw}) that the graph of $f$ is composed of ``chainsaw teeth'', where each tooth is composed of a ``left blade'' and (with the exception of the right-most tooth which contains the number 1) a ``right blade'': the apex of  tooth  number $k\in\N$ is the point $(1/k,1/k)$, the left blade is the line connecting the point  $(2/(2k+1),0)\in \R^2$ with this apex, and the right blade is the line segment connecting the apex with the point $(2/(2k-1),0)$.

  Second, it must be that $(x_0,f(x_0))$ and $(y_0,f(y_0))$ are located on the same blade. To see that  this claim holds, suppose to the contrary that these points are on different blades. We claim that in this case there is a point $z_0\in [0,1]$ (actually many points) satisfying both $x_0<z_0<y_0$ and either $|f(y_0)-f(z_0)|\geq |f(y_0)-f(x_0)|$ or $|f(x_0)-f(z_0)|\geq |f(y_0)-f(x_0)|$. Once this claim is proved (done in the next paragraph), we obtain a contradiction  to the assumption that $(x_0,y_0)$ is a minimizer of $d_X$ on $A_{\epsilon}$ because in the first case $|f(y_0)-f(z_0)|\geq |f(y_0)-f(x_0)|\geq \epsilon$ (since $(x_0,y_0)\in A_{\epsilon}$) and $d_X(y_0,z_0)=y_0-z_0<y_0-x_0=d_X(x_0,y_0)$, and in the second case $|f(x_0)-f(z_0)|\geq |f(y_0)-f(x_0)|\geq \epsilon$ and $d_X(x_0,z_0)=z_0-x_0<y_0-x_0=d_X(x_0,y_0)$. 
	
	Now we prove that there exists such a point $z_0$ with the required properties. Since it is assumed that $x_0$ and $y_0$ are on different blades, since $1/k_0<y_0\leq 1/(k_0-1)$ and since $x_0<y_0$, it follows from \beqref{eq:Chainsaw} that there can be two cases: either $(y_0,f(y_0))$ is on the blade with ``base'' $[1/k_0,2/(2k_0-1)]$ and $(x_0,f(x_0))$ is on a blade located to the left of this blade, and thus $x_0<1/k_0<y_0$, or $(y_0,f(y_0))$ is on the blade with base  $[2/(2k_0-1),1/(k_0-1)]$ and $(x_0,f(x_0))$ is on a blade located to the left of this blade, and so $x_0<2/(2k_0-1)\leq y_0$. In the first case we can take $z_0:=1/k_0$. Indeed, $x_0<z_0<y_0$; in addition, since $f$ attains its maximal value on $[0,2/(2k_0-1)]$ at the point $z_0$, if $f(x_0)\geq f(y_0)$, then $|f(y_0)-f(z_0)|=f(z_0)-f(y_0)\geq f(x_0)-f(y_0)=|f(y_0)-f(x_0)|$, and if $f(x_0)<f(y_0)$, then $|f(x_0)-f(z_0)|=f(z_0)-f(x_0)\geq f(y_0)-f(x_0)=|f(y_0)-f(x_0)|$. In the second case, if $y_0>2/(2k_0-1)$, then  we can take  $z_0:=2/(2k_0-1)$ since in this case $x_0<z_0<y_0$ and $f(z_0)=0$, and either we have $f(y_0)\geq f(x_0)$, and then $|f(y_0)-f(z_0)|=f(y_0)\geq f(y_0)-f(x_0)=|f(y_0)-f(x_0)|$, or we have $f(x_0)\geq f(y_0)$ and then $|f(x_0)-f(z_0)|=f(x_0)\geq f(x_0)-f(y_0)=|f(y_0)-f(x_0)|$. Finally, if $y_0=2/(2k_0-1)$, then $x_0$ must be smaller than $1/k_0$ (otherwise both $(x_0,f(x_0))$ and $(y_0,f(y_0))$ are located on the blade with base $[1/k_0,2/(2k_0-1)]$), and therefore this case reduces to the  first case in which we take $z_0:=1/k_0$. 
  
  So the assumption that $(x_0,f(x_0))$ and $(y_0,f(y_0))$ are located on  different blades leads to the existence of the above mentioned point $z_0$, which by itself leads to a contradiction. Hence this proves that $(x_0,f(x_0))$ and $(y_0,f(y_0))$ are located on the same blade. But then either this is the right blade of tooth number $k_0$, so \beqref{eq:Chainsaw} implies that $|f(x_0)-f(y_0)|=(2k_0-1)(y_0-x_0)=(2k_0-1)d_X(x_0,y_0)$, or this is the left blade of tooth number $k_0-1$ and then again  \beqref{eq:Chainsaw} implies that $|f(x_0)-f(y_0)|=(2k_0-1)(y_0-x_0)=(2k_0-1)d_X(x_0,y_0)$. Since $(x_0,y_0)\in A_{\epsilon}$, we  know that $|f(x_0)-f(y_0)|\geq \epsilon$. Now we combine this inequality with the previous lines, with \beqref{eq:delta_f(1/n)}, with the  assumption that $1/n<\epsilon$, and with \beqref{n>=k_0}, and obtain the desired conclusion:
\begin{multline*}
\delta_f(\epsilon)=d_X(x_0,y_0)=\frac{|f(x_0)-f(y_0)|}{2k_0-1}\geq \frac{\epsilon}{2k_0-1}
>\frac{1}{n(2n-1)}>\frac{1}{(2n+1)n}\geq \delta_f\left(\frac{1}{n}\right).
\end{multline*}
\end{proof}

{\noindent }\textbf{Acknowledgments: So Long, and Thanks for All the Involved Entities}\vspace{0.1cm}\\ 
This article is the result of a long and challenging trek which started in 2007. Although most of  the work on the article has been done in years  in which I have been associated with The  Technion, Haifa, Israel (2007-2010, 2016, 2018--2019), other stations in space and time have  benefited me regarding the article: the University of Haifa, Haifa, Israel (2010), the National Institute of Pure and Applied Mathematics (IMPA), Rio de Janeiro, Brazil (2012), and the Institute of Mathematical and Computer Sciences (ICMC), University of S\~ao Paulo, S\~ao Carlos, Brazil (2015). In addition, I would like to use this opportunity to thank several people, especially  Gregory Shapiro for a useful discussion regarding Example \bref{ex:delta_f},  Zbigniew H. Nitecki for a useful discussion regarding \cite{Nitecki2009book}, Jose M. Almira for a useful discussion regarding  \cite{AlmiraPassot2008jour}, and Jamanadas R. Patadia for useful remarks on the structure of the article. 

\bibliographystyle{acm}
\bibliography{biblio}
\end{document}